\documentclass[pdflatex,sn-mathphys-num]{sn-jnl}

\usepackage[T1]{fontenc}%
\usepackage{lmodern}%
\usepackage{graphicx}%
\usepackage{multirow}%
\usepackage{amsmath,amssymb,amsfonts}%
\usepackage{amsthm}%
\usepackage{mathrsfs}%
\usepackage[title]{appendix}%
\usepackage{xcolor}%
\usepackage{textcomp}%
\DeclareUnicodeCharacter{2131}{\ensuremath{\mathcal{F}}}
\DeclareUnicodeCharacter{2102}{\ensuremath{\mathbb{C}}}
\DeclareUnicodeCharacter{2500}{-}
\DeclareUnicodeCharacter{2550}{=}
\DeclareUnicodeCharacter{00AC}{\ensuremath{\lnot}}
\DeclareUnicodeCharacter{2286}{\ensuremath{\subseteq}}
\DeclareUnicodeCharacter{22A5}{\ensuremath{\bot}}
\DeclareUnicodeCharacter{0141}{\L{}}
\DeclareUnicodeCharacter{0142}{\l{}}
\DeclareUnicodeCharacter{0393}{\ensuremath{\Gamma}}
\DeclareUnicodeCharacter{0394}{\ensuremath{\Delta}}
\DeclareUnicodeCharacter{2014}{\textemdash}
\DeclareUnicodeCharacter{2192}{\ensuremath{\to}}
\DeclareUnicodeCharacter{22A2}{\ensuremath{\vdash}}
\DeclareUnicodeCharacter{22AC}{\ensuremath{\nvdash}}
\DeclareUnicodeCharacter{1D40C}{\ensuremath{\mathbf{M}}}
\usepackage{manyfoot}%
\usepackage{booktabs}%

\usepackage{bussproofs}

\providecommand{\nvdash}{\not\vdash}

\theoremstyle{thmstyleone}%
\newtheorem{theorem}{Theorem}%
\newtheorem{proposition}[theorem]{Proposition}%

\theoremstyle{thmstyletwo}%
\newtheorem{remark}{Remark}%

\theoremstyle{thmstylethree}%

\usepackage{tcolorbox}
\tcbuselibrary{breakable}
\usepackage{fancyvrb}

\newtcolorbox{coqbox}{
  breakable,
  colback=gray!10,
  colframe=gray!40,
  boxrule=0.4pt,
  arc=1pt,
}

\makeatletter
\renewcommand{\verbatim@font}{\ttfamily\fontsize{8bp}{9.2bp}\selectfont}
\makeatother
\begin{document}

\title[]{A Proof in Coq that Core Logic is
  \textit{not} Paraconsistent}

\author*[1]{\fnm{Joseph} \sur{Vidal-Rosset}}\email{joseph.vidal-rosset@univ-lorraine.fr}

\affil*[1]{\orgdiv{Département de philosophie, Archives Poincaré UMR 7117 CNRS},
  \orgname{Université de Lorraine},
  \orgaddress{\city{Nancy}, \country{France}}}

\abstract{Tennant claims that his Core logic $\mathbb{C}$ is
  \emph{paraconsistent}. It means that the sequent of the First
  Lewis Paradox, i.e.
  $$\lnot A, A \vdash B$$
  is declared false, and its corresponding antisequent, called
  `Claim~1', i.e.
  $$\lnot A, A \nvdash B$$
  true, as in minimal logic $\mathbf{M}$. This paper proves that
  Claim~1 entails a contradiction in $\mathbb{C}$, so that, to
  preserve consistency, the Core logician must reject the claim
  that his system is paraconsistent. The proof is purely logical,
  in four steps within a five-rule fragment $\mathcal{F}$ of
  $\mathbb{C}$ and its refutation system in the sense of
  \L{}ukasiewicz and Goranko; the Appendix certifies every step in
  Coq --- with no axiom assumed and every commitment displayed as a
  named hypothesis --- and the same certification is replayed
  independently in Lean~4.}


\maketitle

\section{Introduction}\label{sec:intro}

In intuitionistic logic (for short \textbf{I}), the famous Lewis
Paradox is derivable thanks to the rule of \emph{Weakening on the
right}:
\begin{prooftree}
\AxiomC{}
\RightLabel{\tiny{$Ax.$}}
\UnaryInfC{$A \vdash A$}
\RightLabel{\tiny{$L\lnot$}}
\UnaryInfC{$\lnot A, A \vdash $}
\RightLabel{\tiny{$R~Wk.$}}
\UnaryInfC{$\lnot A, A \vdash B$}
\end{prooftree}
In \emph{minimal logic} (for short \textbf{M}, a \emph{subsystem} of
\textbf{I}), the rule \emph{Weakening on the left} allows the
derivability of the negated-conclusion form of the First Lewis
Paradox:
\begin{prooftree}
\AxiomC{}
\RightLabel{\tiny{$Ax.$}}
\UnaryInfC{$A \vdash A$}
\RightLabel{\tiny{$L\lnot$}}
\UnaryInfC{$\lnot A, A \vdash $}
\RightLabel{\tiny{$L~Wk.$}}
\UnaryInfC{$B, \lnot A, A \vdash $}
\RightLabel{\tiny{$R\lnot$}}
\UnaryInfC{$\lnot A, A \vdash \lnot B$}
\end{prooftree}

In the system that Tennant invented at the end of the 1980s
\cite{tennant1987}, formerly called `Intuitionistic Relevant Logic'
and now called `Core Logic' (for short, $\mathbb{C}$)
\cite{tennant2017}, \emph{there is no Weakening}, neither on the
right nor on the left, and the relevance of $\mathbb{C}$ is claimed
to operate \emph{only at the level of the turnstile}
\cite[p.~41]{tennant2017}. Indeed, any intuitionistic theorem is
also a Core theorem \cite[p.~43]{tennant2017}:
\begin{quote}
  Core Logic matches Intuitionistic Logic on logical
  \emph{theorems}, on inconsistencies, and on deducibility from
  consistent sets of premises.
\end{quote}
But
\begin{equation}
\lnot A, A \nvdash B \tag{Claim 1}\label{claim1}
\end{equation}
and
\begin{equation}
\lnot A, A \nvdash \lnot B \tag{Claim 2}\label{claim2}
\end{equation}
are \emph{basic antisequents} in $\mathbb{C}$: the conclusions of
the derivations above are claimed by Tennant \emph{always false} in
his system.

\eqref{claim1} is why Tennant declares $\mathbb{C}$ to be, like
\textbf{M}, \emph{paraconsistent} \cite[p.~156]{tennant2017}. By
contrast, \eqref{claim2} is why he denies that \textbf{M} is a
subsystem of $\mathbb{C}$, describing instead an overlap between
\textbf{M} and $\mathbb{C}$ \cite[p.~35]{tennant2017}:

\begin{itemize}
\item $\mathbb{C} \cap \mathbf{M} \neq \emptyset$, i.e., \textbf{M}
  and $\mathbb{C}$ share the proofs of some sequents like, for
  example,
  \begin{equation}
    \label{eq:1}
    A \to B, A \vdash B
  \end{equation}
\item $\mathbb{C} \not\subseteq \mathbf{M}$, because the following
  sequent is provable in $\mathbb{C}$, but not in \textbf{M}:
  \begin{equation}
    \label{eq:2}
    \lnot A \vdash A \to B
  \end{equation}
\item $\mathbf{M} \not\subseteq \mathbb{C}$, because of
  \eqref{claim2}.
\end{itemize}

Section~\ref{sec:theorem} states the theorem and proves it in four
purely logical steps: the fragment and its two readings
(\S\ref{sec:fragment}), the derivability of DNS.1 and DNS.2
(\S\ref{sec:rules}), the invertibility of DNS.1
(\S\ref{sec:invertible}), and the refutation rule
$\overline{DNS.1}$, contrapositive of that invertibility --- which
is also where the theory of \emph{refutation systems}, the
deductive systems for rejected statements initiated by
\L{}ukasiewicz \cite{lukasiewicz1957} and developed for sequents by
Tiomkin \cite{tiomkin1988} and Goranko \cite{goranko1994}, enters
the argument (\S\ref{sec:antidns}) --- before the contradiction
concludes (\S\ref{sec:contradiction}). Section~\ref{sec:corollary}
extends the result to \eqref{claim2}, and
Section~\ref{sec:scope-reason} assesses its scope and its reason.
No proof assistant is needed to read any of this. The complete Coq
certification --- twelve statements, no axiom, every commitment
displayed as a named hypothesis --- is given in
Appendix~\ref{app:coq}, block by block, each block tied to the step
of the proof it certifies, and replayed independently in Lean~4
(Appendix~\ref{secA1}).

\section{Theorem}\label{sec:theorem}

\begin{theorem}\label{thm:main}
The theory according to which $\mathbb{C}$ is paraconsistent is
false, because \eqref{claim1} entails a contradiction in
$\mathbb{C}$. Therefore, to preserve consistency, the Core logician
must reject \eqref{claim1}, that is to say the claim that his
system is paraconsistent.
\end{theorem}

\noindent\emph{Proof in four steps, in $\mathcal{F}$ and in its
refutation system in the sense of \L{}ukasiewicz and Goranko.} The
reasoning below is purely logical and can be read without any proof
assistant; Appendix~\ref{app:coq} certifies each step in Coq,
block by block, and the correspondence between steps and certified
statements is recalled at the head of each block.

\subsection{Fragment $\mathcal{F}$ of $\mathbb{C}$}\label{sec:fragment}

\begin{table}[h]
\caption{Fragment $\mathcal{F}$ of $\mathbb{C}$}\label{tab:fragment}
\begin{tabular}{@{}ccc@{}}
\toprule
 &
\begin{minipage}{0.30\textwidth}\centering
\begin{prooftree}
\AxiomC{}
\RightLabel{\tiny{\textit{Ax.}}}
\UnaryInfC{$A \vdash A$}
\end{prooftree}
\end{minipage}
& \\[1ex]
\begin{minipage}{0.30\textwidth}\centering
\begin{prooftree}
\AxiomC{$\Delta \vdash A$}
\RightLabel{\tiny{$L\lnot$}}
\UnaryInfC{$\lnot A, \Delta \vdash$}
\end{prooftree}
\end{minipage}
& &
\begin{minipage}{0.30\textwidth}\centering
\begin{prooftree}
\AxiomC{$\Delta \vdash B$}
\RightLabel{\tiny{$R\to$}}
\UnaryInfC{$\Delta \backslash \{A\} \vdash A \to B$}
\end{prooftree}
\end{minipage}
  \\[1ex]
  \\
\begin{minipage}{0.30\textwidth}\centering
\begin{prooftree}
\AxiomC{$\Delta \vdash A$}
\AxiomC{$B, \Gamma \vdash C$}
\RightLabel{\tiny{$L\to$}}
\BinaryInfC{$A \to B, \Delta, \Gamma \vdash C$}
\end{prooftree}
\end{minipage}
& &
\begin{minipage}{0.30\textwidth}\centering
\begin{prooftree}
\AxiomC{$A, \Delta \vdash$}
\RightLabel{\tiny{$R\to_{\mathbb{C}}$}}
\UnaryInfC{$\Delta \vdash A \to B$}
\end{prooftree}
\end{minipage}
\\
\botrule
\end{tabular}
\end{table}

Table~\ref{tab:fragment} reproduces a \emph{fragment} $\mathcal{F}$
of $\mathbb{C}$, with rules drawn from Tennant's exposition of
\textbf{M}, \textbf{I} and $\mathbb{C}$
\cite[pp.~159--163]{tennant2017}. The label $R\to_{\mathbb{C}}$ is
mine, not Tennant's; I use it because it is the only
Core-specific rule in $\mathcal{F}$, the other four being
\emph{exactly} the same as in \textbf{M}.

\begin{remark}\label{rema1}
In this syntax, \[\Delta \vdash B\] presupposes that $\Delta$ is
consistent, otherwise \[\Delta \vdash\] must be written.
\end{remark}

\begin{remark}\label{rema2}
In textbooks, rule $R\to$ is usually written as
\begin{prooftree}
\AxiomC{$A, \Delta \vdash B$}
\RightLabel{\tiny{$R\to$}}
\UnaryInfC{$\Delta \vdash A \to B$}
\end{prooftree}
Tennant's preferred format reproduced in Table~\ref{tab:fragment}
for $\mathcal{F}$ means only that discharging an assumption to
derive the conclusion is \emph{optional} with this rule, because
the context is supposed to be consistent. In other words, with
consistent contexts, the rule of Weakening on the left is
\emph{admissible} in $\mathbb{C}$, which explains why
\begin{equation}
  \label{eq:belnap-paradox}
  A \vdash B \to A
\end{equation}
is Core \emph{provable} \cite[p.~35]{tennant2017} as follows
\begin{prooftree}
  \AxiomC{}
  \RightLabel{\tiny{$Ax.$}}
  \UnaryInfC{$B \vdash B$}
  \RightLabel{\tiny{$R\to$}}
  \UnaryInfC{$B \vdash A \to B$}
\end{prooftree}
even if it is  usually considered a paradox from the  point of view of
Anderson         and          Belnap's         relevance         logic
\cite{anderson_belnap_1975}.    Again,   $\mathbb{C}$    is   relevant
\emph{only at  the level of  the turnstile}, hence  \eqref{claim1} and
\eqref{claim2}.
\end{remark}

In what follows, $\mathcal{F}_{\mathbf{M}}$ denotes the fragment
under its \emph{minimal} reading (the four shared rules \emph{Ax.},
$L\lnot$, $R\to$, $L\to$ alone) and $\mathcal{F}_{\mathbb{C}}$ the
same fragment under $\mathbb{C}$'s reading, which adds
$R\to_{\mathbb{C}}$. An antisequent $\Gamma \nvdash C$ says that no
derivation concludes $\Gamma \vdash C$. Following Tennant's own
convention, contexts are treated as \emph{sets of available
assumptions}: the judgment $\Gamma \vdash C$ says that some
subsequent $\Delta \vdash C$, with $\Delta \subseteq \Gamma$, is
Core-provable --- which is why extending the pool of available
assumptions is deductively inert, whereas composing two derivations
(Cut) is a different matter altogether, examined in
Section~\ref{sec:reason}. (Appendix~\ref{app:coq}, Block~A,
encodes the language and the calculus: contexts are treated
extensionally, the left rules locating their principal formula by
membership, and the Core-specific rule is restricted by typing to
$\mathbb{C}$'s reading.)

\noindent\emph{Reply to an anticipated objection.} It might be
objected that the argument of this paper smuggles into $\mathbb{C}$
results which are valid only in minimal logic. The objection misses
the point. The whole argument is conducted inside fragment
$\mathcal{F}$ of $\mathbb{C}$, which by construction shares with
$\mathbf{M}$ the four rules \emph{Ax.}, $L\lnot$, $R\to$, $L\to$:
every derivation in $\mathcal{F}_{\mathbf{M}}$ is, by definition of
a fragment system, also a Core derivation in $\mathcal{F}$ ---
DNS.1 is certified \emph{uniformly} in the two readings for this
very reason (Appendix, Block~B). Nothing is imported from outside
$\mathcal{F}$; the only Core-specific rule, $R\to_{\mathbb{C}}$,
enters the argument exactly where DNS.2 needs it.

\subsection{DNS.1 and DNS.2 are derivable}\label{sec:rules}

`DNS' being an acronym for `Double Negation \`a la Slaney'
\cite{Slaney94minlog}, I call `DNS.1' and `DNS.2' the rules listed
in Table~\ref{tab:dns} below, together with their respective
derivations in $\mathcal{F}$, which prove that they are derivable
in $\mathbb{C}$. (By definition, a rule is derivable in a given
system of rules if its conclusion is derivable from its premisses
within the system \cite[p.~47]{vonplato2013}.)

\begin{table}[h]
\caption{DNS rules and their respective derivations in $\mathcal{F}$}\label{tab:dns}
\begin{tabular}{@{}cc@{}}
\toprule
\begin{minipage}{0.45\textwidth}\centering
\begin{prooftree}
\AxiomC{$A, \Delta \vdash B$}
\RightLabel{\tiny{DNS.1}}
\UnaryInfC{$(A \to B) \to B, \Delta \vdash B$}
\end{prooftree}
\end{minipage}
&
\begin{minipage}{0.45\textwidth}\centering
\begin{prooftree}
\AxiomC{$A, \Delta \vdash$}
\RightLabel{\tiny{DNS.2}}
\UnaryInfC{$(A \to B) \to B, \Delta \vdash B$}
\end{prooftree}
\end{minipage}
  \\[1ex]
  \\
\begin{minipage}{0.45\textwidth}\centering
\begin{prooftree}
\AxiomC{$A, \Delta \vdash B$}
\RightLabel{\tiny{$R\to$}}
\UnaryInfC{$\Delta \vdash A \to B$}
\AxiomC{}
\RightLabel{\tiny{$Ax$}}
\UnaryInfC{$B \vdash B$}
\RightLabel{\tiny{$L\to$}}
\BinaryInfC{$(A \to B) \to B, \Delta \vdash B$}
\end{prooftree}
\end{minipage}
&
\begin{minipage}{0.45\textwidth}\centering
\begin{prooftree}
\AxiomC{$A, \Delta \vdash$}
\RightLabel{\tiny{$R\to_{\mathbb{C}}$}}
\UnaryInfC{$\Delta \vdash A \to B$}
\AxiomC{}
\RightLabel{\tiny{$Ax$}}
\UnaryInfC{$B \vdash B$}
\RightLabel{\tiny{$L\to$}}
\BinaryInfC{$(A \to B) \to B, \Delta \vdash B$}
\end{prooftree}
\end{minipage}
\\
\botrule
\end{tabular}
\end{table}

\noindent (Appendix, Block~B: theorem \texttt{DNS1\_in\_ℱ}, proved
\emph{uniformly} in the two readings of $\mathcal{F}$ since its
derivation uses only the four shared rules, and theorem
\texttt{DNS2\_instantiated}, fed by the inconsistency sequent
\texttt{absurdity\_core} and using $R\to_{\mathbb{C}}$. The only
structural fact involved, \texttt{context\_monotonicity}, is not
Tennant's Weakening --- no such rule exists in $\mathbb{C}$, and
none is added --- and it is neither Cut nor transitivity: no two
derivations are composed, no cut formula exists, no formula is
eliminated. It rebuilds \emph{one} derivation, unchanged and
height-preservingly, in a larger pool of available assumptions ---
the very junction of contexts $(\Delta, \Gamma)$ that Tennant's own
rule $L\to$ performs notationally. Safety runs in both directions:
the monotone encoding derives \emph{more} sequents than Tennant's
strict reading, so every underivability proved below holds
\emph{a fortiori} of the strict reading; and the derivable witness
of the final collision uses every formula of its context, so no
dilution slack is exploited on the positive side.)

\subsection{DNS.1 is invertible}\label{sec:invertible}

A rule is \emph{invertible} when, from the derivability of a
sequent of the form of its conclusion, the derivability of all its
premisses follows \cite[p.~93]{platonegri2011}. DNS.1 is provably
\emph{invertible}: outright in $\mathcal{F}_{\mathbf{M}}$, as
certified below, and in $\mathcal{F}$ as a whole under Tennant's own
consistency proviso for contexts (Remark~\ref{rema1}).

Reading the syntactic  derivation of DNS.1 (i.e.,  the left derivation
in Table~\ref{tab:dns}) from root to top,  the last rule to derive the
premiss of DNS.1 is $R\to$. Therefore,  if $R\to$ is invertible, so is
DNS.1. The following proof of the  invertibility of $R\to$ is based on
the  principle of  structural induction  on derivations,  used by  von
Plato \cite[pp.~68--70]{vonplato2013} to prove  the inversion lemma in
which  the invertibility  of  $R\to$ is  established.  The reader  not
concerned  with  details  may  skip   this  section:  because  in  his
discussion of the  converse of the Deduction  Theorem, Tennant himself
acknowledges         the        invertibility         of        $R\to$
\cite[p.~46]{tennant2017}.  The  invertibility  of DNS.1  is  a  basic
logical  truth:  provably valid  in  \textbf{M},  DNS.1 containing  no
negation, it cannot be invalid in $\mathbb{C}$.

\subsubsection{Syntactic Proof}

The proof proceeds by case analysis on the rule at the root of
$\pi$ --- that is, the rule whose conclusion is the endsequent
itself. We assume throughout that $\Delta$ is consistent; the case
where the root rule is $R\to_{\mathbb{C}}$ is therefore excluded by
Tennant's own caveat (Remark~\ref{rema1} above).

Two preliminary points are essential to the analysis:

\begin{enumerate}
\item[(i)] \label{point_i} The only initial sequents of
  $\mathbb{C}$ are of the form $A \vdash A$ where $A$ is atomic,
  since \emph{Weakening on the left} is not a rule of the sequent
  calculus for $\mathbb{C}$.

\item[(ii)] \label{point_ii} The form of Tennant's rule $R\to$ is
\begin{prooftree}
\AxiomC{$\Delta \vdash B$}
\RightLabel{\tiny{$R\to$}}
\UnaryInfC{$\Delta \backslash \{A\} \vdash A \to B$}
\end{prooftree}
where the  discharge of  $A$ is optional:  it may  be \emph{effective}
($A \in \Delta$ and $A$ is removed from the context in the conclusion)
or \emph{vacuous} ($A  \notin \Delta$, and the  rule simply introduces
$A  \to  B$  without  affecting  the context).  The  vacuous  case  is
precisely what licenses the derivation of \eqref{eq:2} in $\mathbb{C}$
as recalled  in Remark~\ref{rema2} above. This  unusual formulation of
$R\to$  --- where  the  antecedent of  the  discharged implication  is
\emph{subtracted} from the context rather than added to it --- is what
makes invertibility almost trivial:  it essentially amounts to putting
the antecedent of the conditional back into context $\Delta$.
\end{enumerate}

We can now state the inversion lemma for $R\to$ in $\mathcal{F}$:
\[\Delta \backslash \{A\} \vdash A \to B
~~\Rightarrow~~
A, \Delta \backslash \{A\} \vdash B\]
whether $A$ was effectively discharged from $\Delta$, or vacuously
introduced by $R\to$.

\begin{itemize}
\item \emph{Base case: $\pi$ reduces to the application of $R\to$
    to an atomic axiom.} The shortest derivation of a sequent of
  the form $\Delta \backslash \{A\} \vdash A \to B$ in $\mathbb{C}$
  is
  \begin{prooftree}
  \AxiomC{}
  \RightLabel{\tiny{\textit{Ax.}}}
  \UnaryInfC{$A \vdash A$}
  \RightLabel{\tiny{$R\to$}}
  \UnaryInfC{$\vdash A \to A$}
  \end{prooftree}
  obtained by effective discharge of $A$ from the atomic axiom $A
  \vdash A$.

  In this case, $B = A$, $\Delta = \{A\}$, $\Delta \backslash \{A\}
  = \emptyset$. The inversion of $R\to$ recovers the initial
  sequent $A \vdash A$ itself.

  \emph{Illustrative example.} A slightly longer derivation,
  instructive for the non-Core-logic reader, produces the identity
  sequent on an implication. In $\mathbb{C}$, $A \to B \vdash A \to
  B$ is \emph{not} an axiom (axioms are atomic); it is derived
  from atomic axioms by:
  \begin{prooftree}
  \AxiomC{}
  \RightLabel{\tiny{\textit{Ax.}}}
  \UnaryInfC{$A \vdash A$}
  \AxiomC{}
  \RightLabel{\tiny{\textit{Ax.}}}
  \UnaryInfC{$B \vdash B$}
  \RightLabel{\tiny{$L\to$}}
  \BinaryInfC{$A, A \to B \vdash B$}
  \RightLabel{\tiny{$R\to$}}
  \UnaryInfC{$A \to B \vdash A \to B$}
  \end{prooftree}
  The sub-derivation up to $L\to$ establishes $A, A \to B \vdash
  B$, which is the inversion of the endsequent. The final $R\to$
  step, with effective discharge of $A$, recovers the endsequent
  $A \to B \vdash A \to B$.

\item \emph{Inductive case: the root rule is $L\to$.} In this
  case, $A \to B$ is propagated unchanged from the right premiss
  of $L\to$ to the conclusion. The derivation $\pi$ ends with
  \begin{prooftree}
  \AxiomC{$\vdots$}
  \UnaryInfC{$\Delta_1 \vdash C$}
  \AxiomC{$\vdots$}
  \UnaryInfC{$D, \Delta_2 \vdash A \to B$}
  \RightLabel{\tiny{$L\to$}}
  \BinaryInfC{$C \to D, \Delta_1, \Delta_2 \vdash A \to B$}
  \end{prooftree}
  where $\Delta \backslash \{A\} = C \to D, \Delta_1, \Delta_2$.

  The right sub-derivation establishes $D, \Delta_2 \vdash A \to
  B$, a sequent of the very form covered by the present lemma. By
  the same case analysis applied to this sub-derivation, we
  obtain a derivation of $A, D, \Delta_2 \vdash B$, where $A$ is
  added to the context exactly as licensed by Tennant's form of
  $R\to$ (point~(ii) above). Reapplying $L\to$ with this new right
  premiss:
  \begin{prooftree}
  \AxiomC{$\vdots$}
  \UnaryInfC{$\Delta_1 \vdash C$}
  \AxiomC{$\vdots$}
  \UnaryInfC{$A, D, \Delta_2 \vdash B$}
  \RightLabel{\tiny{$L\to$}}
  \BinaryInfC{$A, C \to D, \Delta_1, \Delta_2 \vdash B$}
  \end{prooftree}
  i.e., $A, \Delta \vdash B$, as required. The context is
  propagated through $L\to$ unchanged.

\item \emph{Excluded case: the root rule is $L\lnot$.} This rule
  produces a conclusion of the form $\lnot C, \Delta \vdash$, with
  \emph{empty succedent}. It cannot produce $\Delta \backslash
  \{A\} \vdash A \to B$. This case does not arise.

\item \emph{Excluded case: the root rule is $R\to_{\mathbb{C}}$.}
  This rule has the form
  \begin{prooftree}
  \AxiomC{$A, \Delta \vdash$}
  \RightLabel{\tiny{$R\to_{\mathbb{C}}$}}
  \UnaryInfC{$\Delta \vdash A \to B$}
  \end{prooftree}
  The premiss $A, \Delta \vdash$ asserts that $A, \Delta$ is
  \emph{inconsistent}. This case is excluded by Tennant's
  consistency condition: $R\to_{\mathbb{C}}$ lies outside the
  scope of the present lemma, by his own caveat.
\end{itemize}

The cases  are exhaustive in  $\mathcal{F}$. Hence $R\to$  is provably
invertible in $\mathcal{F}$. Reading the syntactic derivation of DNS.1
from root  to top,  the last rule  to derive the  premiss of  DNS.1 is
$R\to$,    and    so    DNS.1   inherits    the    invertibility    of
$R\to$. \hfill$\blacksquare$

\subsubsection{Semantic proof}

From a semantic point of view, the syntactic invertibility is
translated by the fact that all possible distributions of truth
values on $A$ and $B$ in
\begin{prooftree}
\AxiomC{$A, \Delta \vdash B$}
\RightLabel{\tiny{DNS.1}}
\UnaryInfC{$(A \to B) \to B, \Delta \vdash B$}
\end{prooftree}
make premiss and conclusion \emph{equivalent}. The proof is
immediate in one fell swoop \emph{\`a la Quine}
\cite[pp.~45--52]{quine1982}. If $B$ is $\top$, then for both
possible values of $A$, premiss and conclusion of the rule reduce
to $\top$, hence are trivially equivalent. To disprove the
invertibility of DNS.1, $B$ must be $\bot$. But if $A$ is $\bot$,
we get
\begin{prooftree}
\AxiomC{$\bot, \Delta \vdash \bot$}
\RightLabel{\tiny{DNS.1}}
\UnaryInfC{$(\bot \to \bot) \to \bot, \Delta \vdash \bot$}
\end{prooftree}
i.e.,
\begin{prooftree}
\AxiomC{$\bot, \Delta \vdash \bot$}
\RightLabel{\tiny{DNS.1}}
\UnaryInfC{$\bot, \Delta \vdash \bot$}
\end{prooftree}
and if $A$ is $\top$:
\begin{prooftree}
\AxiomC{$\top, \Delta \vdash \bot$}
\RightLabel{\tiny{DNS.1}}
\UnaryInfC{$(\top \to \bot) \to \bot, \Delta \vdash \bot$}
\end{prooftree}
i.e.,
\begin{prooftree}
\AxiomC{$\top, \Delta \vdash \bot$}
\RightLabel{\tiny{DNS.1}}
\UnaryInfC{$\top, \Delta \vdash \bot$}
\end{prooftree}
Therefore, in all cases, premiss and conclusion are
\emph{equivalent}, which proves semantically that DNS.1 is
invertible.

\medskip
\noindent\emph{Two important points about this subproof}:
\begin{enumerate}
\item[(a)] Core logic \emph{respecting truth tables}
  \cite[p.~185]{tennant2017}, the Core logician cannot deny without
  contradiction that DNS.1 is invertible. This is \emph{not} a
  question of completeness, but of soundness.

\item[(b)] This semantic proof is more direct than the syntactic
  one. Note that it remains valid under the \emph{minimal logic}
  reading that treats $\bot$ as, essentially, a \emph{nonlogical}
  constant \cite[p.~89]{miller_nadathur_2012}, in which DNS.1 is
  provably invertible. The semantic proof therefore extends to
  $\mathbb{C}$. \hfill$\blacksquare$
\end{enumerate}

\noindent (Appendix, Block~C. The mechanisation establishes the
invertibility by \emph{one} structural induction on derivations ---
the inversion lemma \texttt{DNS1\_inversion\_lemma}, carried under
an invariant on contexts included in
$\{A,\ \lnot A,\ (A\to B)\to B\}$, since the extensional left rules
keep the principal implication available in the premisses and let
contexts grow. Its corollaries state the invertibility of DNS.1 at
the decisive instance
(\texttt{DNS1\_invertible\_at\_decisive\_instance\_in\_ℱ\_M}) and
the underivability, in $\mathcal{F}_{\mathbf{M}}$, of both the
premiss and the conclusion of the instance --- so that, at the
decisive instance, the equivalence of the semantic proof holds in
$\mathcal{F}_{\mathbf{M}}$ with both sides false, and the
load-bearing induction is the inversion lemma itself. The case
$R\to_{\mathbb{C}}$ is excluded by typing --- the mechanical
counterpart of Remark~\ref{rema1}. No ex falso principle is at
work anywhere: the object calculus has no $\bot$ and no such rule,
and the refutations of the mechanisation end in case analyses with
zero cases --- the empty recursor of an inductive type with no
constructor, not an axiom.)

\subsection{Deduction of $\overline{DNS.1}$}\label{sec:antidns}

The invertibility of DNS.1 means that whenever the conclusion is
derivable, the premiss is derivable too. By contraposition,
whenever the premiss is \emph{not} derivable, the conclusion is
\emph{not} derivable either. This is precisely rule
$\overline{DNS.1}$:
\begin{prooftree}
\AxiomC{$A, \Delta \nvdash B$}
\RightLabel{\tiny{$\overline{DNS.1}$}}
\UnaryInfC{$(A \to B) \to B, \Delta \nvdash B$}
\end{prooftree}

In the terminology of \emph{refutation systems} --- the deductive
systems for \emph{rejected} statements initiated by \L{}ukasiewicz
\cite{lukasiewicz1957} and developed for sequents by Tiomkin
\cite{tiomkin1988} and Goranko \cite{goranko1994}; see
\cite{goranko2020} for a survey --- $\overline{DNS.1}$ is a
\emph{refutation rule}, and its licence is exactly Goranko's
correctness discipline for antisequent calculi: a refutation rule
is correct when its converse is a correct rule of the system
\cite[Theorem~2.1]{goranko1994}. The converse of
$\overline{DNS.1}$ is the invertibility of DNS.1, established in
Section~\ref{sec:invertible}. $\overline{DNS.1}$ is therefore not a
meta-rule imported into the kernel from outside: it is a rule
derivable from the kernel's own invertibility, dormant in the
shadow of the system.

A rejection assertion, in \L{}ukasiewicz's sense, has inferential
content only inside a refutation system. The smallest refutation
system that the kernel of $\mathbb{C}$ licenses has \eqref{claim1}
as its only rejection axiom and $\overline{DNS.1}$ as its only
refutation rule; call it $\mathsf{R}$. Two verdicts about
$\mathsf{R}$ are certified (Appendix, Blocks~D and~E). The first is
that $\mathsf{R}$ is \emph{\L{}-correct} for
$\mathcal{F}_{\mathbf{M}}$: everything $\mathsf{R}$ rejects is
underivable in the minimal reading --- in
$\mathcal{F}_{\mathbf{M}}$, $\overline{DNS.1}$ is a metatheorem,
proved as the contrapositive of the certified invertibility
(Appendix, Block~D, theorems \texttt{anti\_DNS1\_holds\_in\_ℱ\_M}
and \texttt{refutation\_system\_Ł\_correct\_for\_ℱ\_M}). The second
verdict is delivered by the next step.

\subsection{Deduction of the contradiction}\label{sec:contradiction}

This point is the conclusion of the proof: as premiss of
$\overline{DNS.1}$, Claim 1 entails an antisequent which is in
contradiction with a consequence of DNS.2:
\begin{prooftree}
\AxiomC{}
\RightLabel{\tiny{(Claim 1)}}
\UnaryInfC{$\lnot A, A \nvdash B$}
\RightLabel{\tiny{$\overline{DNS.1}$}}
\UnaryInfC{$\lnot A, (A \to B) \to B \nvdash B$}
\AxiomC{}
\RightLabel{\tiny{\emph{Ax.}}}
\UnaryInfC{$A \vdash A$}
\RightLabel{\tiny{$L\lnot$}}
\UnaryInfC{$\lnot A, A \vdash$}
\RightLabel{\tiny{DNS.2}}
\UnaryInfC{$\lnot A, (A \to B) \to B \vdash B$}
\BinaryInfC{$\bot$}
\end{prooftree}
Therefore, \eqref{claim1} cannot be maintained in $\mathbb{C}$
without contradiction: if Core logic is consistent, it cannot be
paraconsistent. \hfill$\blacksquare$

In the vocabulary of the previous section, this collision is the
second certified verdict about $\mathsf{R}$: the same refutation
system is \emph{\L{}-incorrect} for $\mathcal{F}_{\mathbb{C}}$ ---
it rejects a sequent that $\mathcal{F}_{\mathbb{C}}$ derives, the
witness being produced by $R\to_{\mathbb{C}}$ alone (Appendix,
Block~E, theorem
\texttt{refutation\_system\_Ł\_incorrect\_for\_ℱ\_ℂ}).

The certified form of the Theorem is a \emph{pure conditional},
resting on no axiom at all: two named hypotheses --- Claim~1
itself, restricted to distinct atoms so that the axiom rule cannot
trivialise it, as Tennant posits it \cite[p.~156]{tennant2017}, and
\texttt{anti\_DNS1\_rule\_for\_ℂ}, the commitment that the kernel's
refutation rule $\overline{DNS.1}$ governs $\mathbb{C}$'s rejection
assertion at this single sequent --- jointly entail \texttt{False}
(Appendix, Block~E, theorem \texttt{claim1\_false}, with a closed
instance at the concrete atoms $0$ and $1$). Nothing foreign to
Core is involved anywhere, since every rule of
$\mathcal{F}_{\mathbf{M}}$ is a rule of $\mathbb{C}$;
Section~\ref{sec:scope-reason} below comments on this commitment.

\noindent\emph{Reply to an anticipated objection.} It might be
objected that displaying \texttt{anti\_DNS1\_rule\_for\_ℂ} as a
hypothesis of the final theorem weakens, or even invalidates, the
proof, by ``forcing'' the conclusion; and, more precisely, that
invertibility is an admissibility property which need not survive
the addition of a new rule, so that the invertibility of DNS.1,
established in $\mathcal{F}_{\mathbf{M}}$, may not be transferred
to $\mathcal{F}_{\mathbb{C}}$ --- indeed $R\to_{\mathbb{C}}$
produces a new derivation at the decisive instance.

The objection is correct --- and it is a theorem of this paper, not
an objection to it. The certification settles the status of the
hypothesis completely, on both sides (Appendix, Block~F). In the
minimal reading, $\overline{DNS.1}$ is a metatheorem, proved
outright. In the Core reading, the formalisation is charitable to
Tennant --- the fragment verifies Claim~1 itself
(\texttt{claim1\_holds\_in\_ℱ\_ℂ}) --- and yet, DNS.2 being
derivable through $R\to_{\mathbb{C}}$, it refutes the commitment
(\texttt{anti\_DNS1\_Ł\_incorrect\_for\_ℱ\_ℂ}; named
\texttt{ℱ\_ℂ\_not\_conservative\_at\_DNS1} in Versions~4 and~5 of
the file): the converse of $\overline{DNS.1}$, the invertibility
of DNS.1, does not survive the passage from
$\mathcal{F}_{\mathbf{M}}$ to $\mathcal{F}_{\mathbb{C}}$ --- the
rule-level form of the \L{}-incorrectness of $\mathsf{R}$
certified in Block~E, and the exact refutation, inside the
calculus, of the displayed hypothesis.

What the objection cannot do is save the paraconsistency claim. A
rejection assertion has inferential content only inside a
refutation system. If the Core logician rejects the hypothesis, as
the objection invites him to, then no refutation discipline
groundable in his own kernel governs \eqref{claim1}: the
metatheoretic properties of the rules that $\mathbb{C}$ shares with
$\mathbf{M}$ no longer bind $\mathbb{C}$'s rejection assertions,
the overlap of \cite[p.~35]{tennant2017} does no justificatory
work, and paraconsistency ``at the level of the turnstile'' is
proclaimed rather than grounded --- while the only candidate
discipline, once applied, is certifiably \L{}-incorrect for his
system. If he accepts the hypothesis, the theorem derives the
contradiction. Either way, \eqref{claim1} cannot be maintained as
Tennant asserts it. And he is committed to the hypothesis twice
over: by the overlap of \cite[p.~35]{tennant2017}, which places
$\mathbb{C}$ with $\mathbf{M}$ on the negative side of Lewis's
sequent --- asserting \eqref{claim1} ``as in $\mathbf{M}$'' just is
claiming the antisequent economy of the shared kernel --- and by
the invertibility of $R\to$ acknowledged in
\cite[p.~46]{tennant2017}, from which $\overline{DNS.1}$ follows.
The necessity in which the formalising logician finds himself of
displaying a commitment that the Core logician endorses in print
but cannot sustain in his own calculus is not an artifice of the
formalisation: it is the mechanical expression of the contradiction
stated by Theorem~\ref{thm:main}.

\section{Corollary}\label{sec:corollary}

\begin{proposition}\label{cor:claim2}
\eqref{claim2} is false; therefore it cannot establish correctly
that $\mathbb{C}$ overlaps minimal logic.
\end{proposition}

\noindent\emph{Proof.} \eqref{claim2} meets exactly the same
contradiction:
\begin{prooftree}
\AxiomC{}
\RightLabel{\tiny{(Claim 2)}}
\UnaryInfC{$\lnot A, A \nvdash \lnot B$}
\RightLabel{\tiny{$\overline{DNS.1}$}}
\UnaryInfC{$\lnot A, (A \to \lnot B) \to \lnot B \nvdash \lnot B$}
\AxiomC{}
\RightLabel{\tiny{\emph{Ax.}}}
\UnaryInfC{$A \vdash A$}
\RightLabel{\tiny{$L\lnot$}}
\UnaryInfC{$\lnot A, A \vdash$}
\RightLabel{\tiny{DNS.2}}
\UnaryInfC{$\lnot A, (A \to \lnot B) \to \lnot B \vdash \lnot B$}
\BinaryInfC{$\bot$}
\end{prooftree}
Therefore, because \eqref{claim2} entails the same contradiction,
it is false, and it cannot establish correctly that $\mathbb{C}$
overlaps minimal logic. \hfill$\blacksquare$

\noindent\emph{Certified counterpart.} The certified proof of
Claim 2's collapse is obtained from \texttt{claim1\_false} by
substituting $\lnot B$ for $B$ throughout the statements and proofs
of \texttt{DNS1\_in\_ℱ}, \texttt{DNS2\_instantiated},
\texttt{DNS1\_inversion\_lemma},
\texttt{DNS1\_conclusion\_underivable\_in\_ℱ\_M},
\texttt{anti\_DNS1\_holds\_in\_ℱ\_M}, and the hypothesis
\texttt{Claim1\_Tennant}. The substitution is purely textual; the
proof scripts are identical modulo this renaming. We omit the
explicit re-instantiation.

\section{Scope and reason of this contradiction}\label{sec:scope-reason}

\subsection{Scope}\label{sec:scope}

This contradiction does \emph{not} mean that Core logic is
inconsistent, but more precisely that \eqref{claim1} cannot be
sustained in this logical system as Tennant asserts it, and that the
same conclusion can be drawn about \eqref{claim2}. Every deduction has been made
\emph{at the level of the turnstile}, justified by the sequent
calculus for Core logic, under the conditions (like consistency)
that Tennant himself accepts. The above Theorem and its Corollary
show that Tennant's two foundational claims about $\mathbb{C}$ ---
its paraconsistency and its merely overlapping relation to minimal
logic --- cannot be sustained.

This proof has been mechanically checked in Coq, block by block, in
Appendix~\ref{app:coq}, with the file
\texttt{core\_logic\_is\_not\_paraconsistent.v} certified by Coq's
kernel without any axiom (\texttt{Print Assumptions} returns
\texttt{Closed under the global context} for every one of its twelve
statements), and independently in Lean~4 (\texttt{\#print axioms}
reports \texttt{[propext]} only; see Appendix~\ref{secA1}). It is logically
flawless. Because it
entails $\overline{DNS.1}$, the invertibility of DNS.1 is the key
point. The rule of \emph{Cut} is Core-admissible in consistent
contexts, and it could have been used to prove that DNS.1 is
invertible in $\mathbb{C}$. The proof by structural induction has
been preferred because it leaves no room for dispute. The status
of \emph{Cut} in inconsistent contexts is another question,
examined in the next section.

Therefore, for the sake of consistency, the Core logician must
abandon at least one of his foundational claims. The next section
diagnoses why this collapse is unavoidable.

\subsection{Reason}\label{sec:reason}

Abandoning \eqref{claim1} would mean losing the identity of
$\mathbb{C}$, which would then collapse into $\mathbf{I}$. And
indeed, for any logician who accepts the rule of \emph{Cut},
rule $R\to_{\mathbb{C}}$ is derivable in $\mathbf{I}$:
\begin{prooftree}
\AxiomC{$A, \Delta \vdash$}
\RightLabel{\tiny{$R~Wk.$}}
\UnaryInfC{$A, \Delta \vdash B$}
\RightLabel{\tiny{$R\to$}}
\UnaryInfC{$\Delta \vdash A \to B$}
\end{prooftree}

Conversely, \emph{Weakening on the right} is derivable in
$\mathbb{C}$ once $R\to_{\mathbb{C}}$ and ordinary Cut are
admitted:
\begin{prooftree}
\AxiomC{$A, \Delta \vdash$}
\RightLabel{\tiny{$R\to_{\mathbb{C}}$}}
\UnaryInfC{$\Delta \vdash A \to B$}
\AxiomC{}
\RightLabel{\tiny{$Ax.$}}
\UnaryInfC{$A \vdash A$}
\AxiomC{}
\RightLabel{\tiny{$Ax.$}}
\UnaryInfC{$B \vdash B$}
\RightLabel{\tiny{$L\to$}}
\BinaryInfC{$A \to B, A \vdash B$}
\RightLabel{\tiny{$Cut$}}
\BinaryInfC{$A, \Delta \vdash B$}
\end{prooftree}

This double derivability shows that $R\to_{\mathbb{C}}$ and
ordinary Cut on the one hand, and $\mathbf{I}$ on the other, are
equi-expressive. The conclusion $\mathbb{C} = \mathbf{I}$ is hard
to escape.

Tennant's response is that the  only Cut admissible in $\mathbb{C}$ is
a \emph{restricted} one  --- call it $\mathbb{C}ut$ ---  which cuts on
the right  as well  as on  the left, in  case of  inconsistent context
\cite[p.   185]{tennant2017}.   With   $\mathbb{C}ut$,  the   previous
derivation ends differently:
\begin{prooftree}
\AxiomC{$A, \Delta \vdash$}
\RightLabel{\tiny{$R\to_{\mathbb{C}}$}}
\UnaryInfC{$\Delta \vdash A \to B$}
\AxiomC{}
\RightLabel{\tiny{$Ax.$}}
\UnaryInfC{$A \vdash A$}
\AxiomC{}
\RightLabel{\tiny{$Ax.$}}
\UnaryInfC{$B \vdash B$}
\RightLabel{\tiny{$L\to$}}
\BinaryInfC{$A \to B, A \vdash B$}
\RightLabel{\tiny{$\mathbb{C}ut$}}
\BinaryInfC{$A, \Delta \vdash$}
\end{prooftree}

The premiss on the left and the conclusion of $\mathbb{C}ut$ are
identical: therefore $\mathbb{C}ut$ is not a \emph{Cut}, but a
\emph{loop}. In computer-science terms, $\mathbb{C}ut$ is a fix
whose only output is its input --- a non-terminating substitute
for the rule it was meant to replace.

The real diagnosis is therefore the following. $R\to_{\mathbb{C}}$
introduces in its conclusion a formula $B$ that occurs
\emph{nowhere} in its premiss $A, \Delta \vdash$. This is the
syntactic signature of a \emph{Cut}: $B$ plays the role of an
eliminated cut-formula. A genuine introduction rule preserves the
subformulas of its premiss in its conclusion; $R\to_{\mathbb{C}}$
does not. As a consequence, the subformula property fails in
$\mathbb{C}$. Recall von Plato's definition
\cite[p.~74]{vonplato2013}:

\begin{quote}
  \textbf{Subformula property for sequent calculus.} Each formula
  in a derivation with the logical rules of sequent calculus is a
  subformula in the endsequent.
\end{quote}

\paragraph{The same phenomenon in natural deduction.} In natural
deduction, $R\to_{\mathbb{C}}$ reads: from a deduction of absurdity
from hypothesis $A$, infer $A \to B$, discharging $A$. Just as, in
sequent calculus, $R\to_{\mathbb{C}}$ compresses Weakening on the
right followed by $R\to$ (as displayed above), in natural deduction
it hides an \emph{ex falso} step inside an introduction. Add this
rule to minimal logic $\mathbf{M}$, where $\bot$ is an ordinary
atom, and \emph{ex falso quodlibet} becomes derivable:
\begin{prooftree}
\AxiomC{$[\bot]^{1}$}
\RightLabel{\tiny{$R\to_{\mathbb{C}}$, 1}}
\UnaryInfC{$\bot \to B$}
\AxiomC{$\bot$}
\RightLabel{\tiny{$\to E$}}
\BinaryInfC{$B$}
\end{prooftree}
so that $\mathbf{M}+R\to_{\mathbb{C}}$ is, deductively,
intuitionistic logic $\mathbf{I}$ --- the natural-deduction
counterpart of the double derivability displayed above; the fact is
elementary enough to be a textbook exercise
(\cite[exercise~4.3, p.~178; solution p.~312]{davidnourraffalli2003}). But look at the shape of this
deduction. The implication $\bot \to B$ is introduced and
immediately eliminated --- a maximal formula --- and the detour is
\emph{irreducible}: the standard reduction would return the
conclusion of the subordinate deduction, which is $\bot$, not $B$;
the $B$ of $R\to_{\mathbb{C}}$ has no ground from which it could be
substituted. Nor can the detour be avoided: the main branch of a
normal deduction of the atom $B$ would consist of eliminations
descending from an assumption, and the only available assumption is
the atom $\bot$, to which no elimination applies. Hence in
$\mathbf{M}+R\to_{\mathbb{C}}$, \emph{ex falso} is derivable but
has \emph{no normal derivation}: normalisation fails. Tennant's
response is, once more, to identify proofs with normal proofs, so
that the deduction above is simply not a proof of $\mathbb{C}$; but
the price is the loss of composition ($\mathbb{C}ut$, examined
above), and the diagnosis is only sharpened. An introduction rule
whose conclusion cannot face its own elimination violates the
inversion principle that governs the harmony of introduction and
elimination rules
\cite[p.~33]{prawitz1965} --- the very principle on which Tennant's
proof-theoretic semantics is built. It will not do to reply that a
logician may stipulate whatever rules he pleases: by Tennant's own
standard, $R\to_{\mathbb{C}}$ is not a meaning-conferring rule; it
is the codification of an irreducible detour --- the very form of an
abnormal derivation, quarantined rather than eliminated.

\paragraph{Mechanical reading of the same diagnosis.} The same
phenomenon receives a precise mechanical formulation in the
certification. The contradiction (Section~\ref{sec:contradiction})
rests on two hypotheses: \texttt{Claim1\_Tennant} (which is
Tennant's own commitment, \cite[p.~156]{tennant2017}) and
\texttt{anti\_DNS1\_rule\_for\_ℂ}, the kernel's refutation rule
stated for $\mathbb{C}$'s reading. The status of the second
hypothesis is itself a mechanically certified witness of
$R\to_{\mathbb{C}}$'s effect. In the minimal reading
$\mathcal{F}_{\mathbf{M}}$ --- the same predicate \emph{stripped
of} the constructor \texttt{R\_arrow\_core} --- the rule is a
metatheorem, proved outright as the contrapositive of a certified
invertibility. In the Core reading $\mathcal{F}_{\mathbb{C}}$ it
fails: DNS.2 makes its conclusion derivable while Claim~1 keeps its
premiss underivable --- the \L{}-incorrectness of
Section~\ref{sec:contradiction}. The commitment is non-trivial
\emph{precisely because} $R\to_{\mathbb{C}}$ is in the system: it
is the single rule that makes the inclusion of derivations
asymmetric, and the contradiction of
Section~\ref{sec:contradiction} traces back exactly to that
asymmetry. The mechanical and proof-theoretic diagnoses thus
converge: $R\to_{\mathbb{C}}$ is the rule that breaks the symmetry
between $\mathbf{M}$ and $\mathbb{C}$ on the fatal sequent, and the
failure of the kernel's refutation discipline at that very sequent
is the mechanical counterpart of the subformula-property failure.

In private correspondence, von Plato pointed out to me that the
subformula  property  --- the  most  famous  consequence of  Gentzen's
\emph{Hauptsatz}  --- gives  a  proof-theoretic  validation of  Kant's
characterisation   of    general   and   pure   logic    as   analytic
\cite[p.~94]{kant2007}.  Therefore, a system in which it fails, and in
which it  cannot be restored  by a genuine  \emph{Cut-elimination}, is
not a logical theory in the proof-theoretic sense.

In conclusion, the Theorem and the present dilemma are one and the
same phenomenon, seen at two levels. The Theorem shows that
\eqref{claim1}, read as a rejection governed by the refutation
discipline of the very kernel that grounds the overlap, entails a
contradiction in $\mathbb{C}$. The present dilemma prices the only
escape: either $\mathbb{C}$ accepts ordinary Cut, in which case it
collapses into $\mathbf{I}$, or it restricts \emph{Cut} to
$\mathbb{C}ut$, in which case $\mathbb{C}ut$ is a loop and the
subformula property fails. Tennant's logic stands on the horns of
this dilemma --- between a system that is no longer paraconsistent,
and a system that is no longer a logic.

\bibliography{paper}


\begin{thebibliography}{9}
\ifx \bisbn   \undefined \def \bisbn  #1{ISBN #1}\fi
\ifx \binits  \undefined \def \binits#1{#1}\fi
\ifx \bauthor  \undefined \def \bauthor#1{#1}\fi
\ifx \batitle  \undefined \def \batitle#1{#1}\fi
\ifx \bjtitle  \undefined \def \bjtitle#1{#1}\fi
\ifx \bvolume  \undefined \def \bvolume#1{\textbf{#1}}\fi
\ifx \byear  \undefined \def \byear#1{#1}\fi
\ifx \bissue  \undefined \def \bissue#1{#1}\fi
\ifx \bfpage  \undefined \def \bfpage#1{#1}\fi
\ifx \blpage  \undefined \def \blpage #1{#1}\fi
\ifx \burl  \undefined \def \burl#1{\textsf{#1}}\fi
\ifx \doiurl  \undefined \def \doiurl#1{\url{https://doi.org/#1}}\fi
\ifx \betal  \undefined \def \betal{\textit{et al.}}\fi
\ifx \binstitute  \undefined \def \binstitute#1{#1}\fi
\ifx \binstitutionaled  \undefined \def \binstitutionaled#1{#1}\fi
\ifx \bctitle  \undefined \def \bctitle#1{#1}\fi
\ifx \beditor  \undefined \def \beditor#1{#1}\fi
\ifx \bpublisher  \undefined \def \bpublisher#1{#1}\fi
\ifx \bbtitle  \undefined \def \bbtitle#1{#1}\fi
\ifx \bedition  \undefined \def \bedition#1{#1}\fi
\ifx \bseriesno  \undefined \def \bseriesno#1{#1}\fi
\ifx \blocation  \undefined \def \blocation#1{#1}\fi
\ifx \bsertitle  \undefined \def \bsertitle#1{#1}\fi
\ifx \bsnm \undefined \def \bsnm#1{#1}\fi
\ifx \bsuffix \undefined \def \bsuffix#1{#1}\fi
\ifx \bparticle \undefined \def \bparticle#1{#1}\fi
\ifx \barticle \undefined \def \barticle#1{#1}\fi
\bibcommenthead
\ifx \bconfdate \undefined \def \bconfdate #1{#1}\fi
\ifx \botherref \undefined \def \botherref #1{#1}\fi
\ifx \url \undefined \def \url#1{\textsf{#1}}\fi
\ifx \bchapter \undefined \def \bchapter#1{#1}\fi
\ifx \bbook \undefined \def \bbook#1{#1}\fi
\ifx \bcomment \undefined \def \bcomment#1{#1}\fi
\ifx \oauthor \undefined \def \oauthor#1{#1}\fi
\ifx \citeauthoryear \undefined \def \citeauthoryear#1{#1}\fi
\ifx \endbibitem  \undefined \def \endbibitem {}\fi
\ifx \bconflocation  \undefined \def \bconflocation#1{#1}\fi
\ifx \arxivurl  \undefined \def \arxivurl#1{\textsf{#1}}\fi
\csname PreBibitemsHook\endcsname

\bibitem[\protect\citeauthoryear{Tennant}{1987}]{tennant1987}
\begin{bbook}
\bauthor{\bsnm{Tennant}, \binits{N.}}:
\bbtitle{Anti-Realism and Logic}.
\bpublisher{Oxford University Press},
\blocation{Oxford}
(\byear{1987})
\end{bbook}
\endbibitem

\bibitem[\protect\citeauthoryear{Tennant}{2017}]{tennant2017}
\begin{bbook}
\bauthor{\bsnm{Tennant}, \binits{N.}}:
\bbtitle{{Core Logic}},
\bedition{1}st edn.
\bpublisher{Oxford University Press},
\blocation{Oxford}
(\byear{2017})
\end{bbook}
\endbibitem

\bibitem[\protect\citeauthoryear{Anderson and
  Belnap}{1975}]{anderson_belnap_1975}
\begin{bbook}
\bauthor{\bsnm{Anderson}, \binits{A.R.}},
\bauthor{\bsnm{Belnap}, \binits{N.D.}}:
\bbtitle{Entailment: The Logic of Relevance and Necessity}
vol. \bseriesno{I}.
\bpublisher{Princeton University Press},
\blocation{Princeton}
(\byear{1975})
\end{bbook}
\endbibitem

\bibitem[\protect\citeauthoryear{Slaney}{1994}]{Slaney94minlog}
\begin{botherref}
\oauthor{\bsnm{Slaney}, \binits{J.}}:
{MINLOG}.
Software
(1994)
\end{botherref}
\endbibitem

\bibitem[\protect\citeauthoryear{{von Plato}}{2013}]{vonplato2013}
\begin{bbook}
\bauthor{\bsnm{{von Plato}}, \binits{J.}}:
\bbtitle{{Elements of Logical Reasoning}}.
\bpublisher{Cambridge University Press},
\blocation{Cambridge}
(\byear{2013})
\end{bbook}
\endbibitem

\bibitem[\protect\citeauthoryear{Negri and {von Plato}}{2011}]{platonegri2011}
\begin{bbook}
\bauthor{\bsnm{Negri}, \binits{S.}},
\bauthor{\bsnm{{von Plato}}, \binits{J.}}:
\bbtitle{{Proof Analysis: A Contribution to {Hilbert}'s Last Problem}}.
\bpublisher{Cambridge University Press},
\blocation{Cambridge}
(\byear{2011})
\end{bbook}
\endbibitem

\bibitem[\protect\citeauthoryear{Quine}{1982}]{quine1982}
\begin{bbook}
\bauthor{\bsnm{Quine}, \binits{W.v.O.}}:
\bbtitle{{Methods of Logic}},
\bedition{4}th edn.
\bpublisher{Harvard University Press},
\blocation{Cambridge, Massachusetts}
(\byear{1982})
\end{bbook}
\endbibitem

\bibitem[\protect\citeauthoryear{Miller and
  Nadathur}{2012}]{miller_nadathur_2012}
\begin{bbook}
\bauthor{\bsnm{Miller}, \binits{D.}},
\bauthor{\bsnm{Nadathur}, \binits{G.}}:
\bbtitle{Programming with Higher-Order Logic}.
\bpublisher{Cambridge University Press},
\blocation{Cambridge}
(\byear{2012})
\end{bbook}
\endbibitem

\bibitem[\protect\citeauthoryear{David, Nour and
  Raffalli}{2003}]{davidnourraffalli2003}
\begin{bbook}
\bauthor{\bsnm{David}, \binits{R.}},
\bauthor{\bsnm{Nour}, \binits{K.}},
\bauthor{\bsnm{Raffalli}, \binits{C.}}:
\bbtitle{{Introduction \`a la logique. Th\'eorie de la
  d\'emonstration}},
\bedition{2}nd edn.
\bpublisher{Dunod},
\blocation{Paris}
(\byear{2003})
\end{bbook}
\endbibitem

\bibitem[\protect\citeauthoryear{Prawitz}{1965}]{prawitz1965}
\begin{bbook}
\bauthor{\bsnm{Prawitz}, \binits{D.}}:
\bbtitle{{Natural Deduction: A Proof-Theoretical Study}},
\bpublisher{Almqvist \& Wiksell},
\blocation{Stockholm}
(\byear{1965})
\end{bbook}
\endbibitem

\bibitem[\protect\citeauthoryear{Kant}{2007}]{kant2007}
\begin{bbook}
\bauthor{\bsnm{Kant}, \binits{I.}}:
\bbtitle{Critique of {Pure} {Reason}},
\bedition{Rev. ed.} edn.
\bpublisher{Penguin Classics},
\blocation{London}
(\byear{2007})
\end{bbook}
\endbibitem


\bibitem[\protect\citeauthoryear{Goranko}{1994}]{goranko1994}
\begin{barticle}
\bauthor{\bsnm{Goranko}, \binits{V.}}:
\batitle{Refutation systems in modal logic}.
\bjtitle{Studia Logica}
\bvolume{53}(\bissue{2}),
\bfpage{299}--\blpage{324}
(\byear{1994})
\end{barticle}
\endbibitem

\bibitem[\protect\citeauthoryear{Goranko, Pulcini and
  Skura}{2020}]{goranko2020}
\begin{botherref}
\oauthor{\bsnm{Goranko}, \binits{V.}},
\oauthor{\bsnm{Pulcini}, \binits{G.}},
\oauthor{\bsnm{Skura}, \binits{T.}}:
Refutation systems: an overview and some applications to
philosophical logics.
In: Liu, F., Ono, H., Yu, J. (eds.) Knowledge, Proof and Dynamics,
pp. 173--197. Springer, Singapore (2020)
\end{botherref}
\endbibitem

\bibitem[\protect\citeauthoryear{\L{}ukasiewicz}{1957}]{lukasiewicz1957}
\begin{bbook}
\bauthor{\bsnm{\L{}ukasiewicz}, \binits{J.}}:
\bbtitle{{Aristotle's Syllogistic from the Standpoint of Modern
  Formal Logic}},
\bedition{2}nd edn.
\bpublisher{Clarendon Press},
\blocation{Oxford}
(\byear{1957})
\end{bbook}
\endbibitem

\bibitem[\protect\citeauthoryear{Tiomkin}{1988}]{tiomkin1988}
\begin{botherref}
\oauthor{\bsnm{Tiomkin}, \binits{M.L.}}:
Proving unprovability.
In: Proceedings of the Third Annual Symposium on Logic in Computer
Science (LICS~'88), pp. 22--26. IEEE Computer Society, Edinburgh
(1988)
\end{botherref}
\endbibitem

\end{thebibliography}

\begin{appendices}

\section{The certified proof}\label{app:coq}

The Coq file \texttt{core\_logic\_is\_not\_paraconsistent.v}
certifies the four-step proof of Section~\ref{sec:theorem} block by
block; the correspondence between blocks and steps is recalled at
the head of each block. The file is reproduced here in full (its
header, which documents this correspondence, is the only part
omitted). Verification: \texttt{Print Assumptions} returns
\texttt{Closed under the global context} for every one of the
twelve certified statements --- nothing is assumed; the two
principles the argument grants to Tennant are explicit hypotheses
of the final theorem, not axioms. The Lean~4 twin of the file
proves the same twelve statements, with \texttt{\#print axioms}
reporting \texttt{[propext]} alone (Appendix~\ref{secA1}).

\subsection*{Block A --- Language and calculus
(Section~\ref{sec:fragment})}

\begin{coqbox}
\begin{verbatim}
(* ══ Block A — Language and calculus ══ *)

From Coq Require Import List ListSet.
Import ListNotations.

Inductive formula : Type :=
  | Var  : nat -> formula
  | Neg  : formula -> formula
  | Impl : formula -> formula -> formula.

Inductive fragment_F : Type :=
  | minimal_F
  | core_logic.

(* [Some A] is a one-formula succedent.
   [None] is the empty succedent.
   Contexts are technically lists, but the left rules locate their
   principal formula extensionally, by membership: no structural
   rule is primitive. An antisequent Γ ⊬ C is rendered directly as
   the type [derivable f G C -> False]. *)

Inductive derivable :
  fragment_F -> set formula -> option formula -> Prop :=

  | Ax :
      forall f G A,
        In A G ->
        derivable f G (Some A)

  | L_neg :
      forall f G A,
        In (Neg A) G ->
        derivable f G (Some A) ->
        derivable f G None

  | R_arrow :
      forall f G A B,
        derivable f (A :: G) (Some B) ->
        derivable f G (Some (Impl A B))

  | L_arrow :
      forall f G A B C,
        In (Impl A B) G ->
        derivable f G (Some A) ->
        derivable f (B :: G) C ->
        derivable f G C

  | R_arrow_core :
      forall G A B,
        derivable core_logic (A :: G) None ->
        derivable core_logic G (Some (Impl A B)).
\end{verbatim}
\end{coqbox}

\subsection*{Block B --- Step 1: DNS.1 and DNS.2 are derivable
(Section~\ref{sec:rules})}

\begin{coqbox}
\begin{verbatim}
(* ══ Block B — Step 1: DNS.1 and DNS.2 are derivable ══ *)

(* Context monotonicity. This lemma is NOT Tennant's Weakening —
   no such rule exists in ℂ, and none is added here — and it is
   NEITHER Cut NOR transitivity: no two derivations are composed,
   no cut formula exists, no formula is eliminated. It rebuilds ONE
   derivation, unchanged and height-preservingly, in a larger pool
   of available assumptions. Its ground is the set convention on
   contexts: [derivable f G C] says that some subsequent Δ ⊆ G is
   Tennant-derivable ⊢ C, so extending the pool is deductively
   inert — it is the very junction of contexts (Δ, Γ) that
   Tennant's own rule L→ performs notationally when it introduces
   its principal formula beside the contexts of its premisses.
   Safety is two-directional: the monotone encoding derives MORE
   sequents than Tennant's strict reading, so every underivability
   proved below holds a fortiori of the strict reading; and the
   derivable witness of the collision uses every formula of its
   context, so no dilution slack is exploited on the positive side.
   Proof by induction on the derivation. *)

Lemma context_monotonicity :
  forall f G G' C,
    derivable f G C ->
    (forall A, In A G -> In A G') ->
    derivable f G' C.
Proof.
  intros f G G' C HD.
  revert G'.
  induction HD; intros G' Hsub.
  - apply Ax. apply Hsub. exact H.
  - apply L_neg with (A := A).
    + apply Hsub. exact H.
    + apply IHHD. exact Hsub.
  - apply R_arrow.
    apply IHHD.
    intros X HX. simpl in HX |- *.
    destruct HX as [HX | HX].
    + left. exact HX.
    + right. apply Hsub. exact HX.
  - apply L_arrow with (A := A) (B := B).
    + apply Hsub. exact H.
    + apply IHHD1. exact Hsub.
    + apply IHHD2.
      intros X HX. simpl in HX |- *.
      destruct HX as [HX | HX].
      * left. exact HX.
      * right. apply Hsub. exact HX.
  - apply R_arrow_core.
    apply IHHD.
    intros X HX. simpl in HX |- *.
    destruct HX as [HX | HX].
    + left. exact HX.
    + right. apply Hsub. exact HX.
Qed.

(* The inconsistency sequent ¬A, A ⊢ that feeds DNS.2, immediate
   since L¬ locates ¬A by membership. *)

Lemma absurdity_core :
  forall a : nat,
    derivable core_logic [Var a; Neg (Var a)] None.
Proof.
  intro a.
  apply L_neg with (A := Var a).
  - simpl. right. left. reflexivity.
  - apply Ax. simpl. left. reflexivity.
Qed.

(* DNS.1, proved uniformly in the fragment indicator f, hence at
   once in ℱ_𝐌 and in ℱ_ℂ: its derivation uses only the four shared
   rules and context monotonicity. *)

Theorem DNS1_in_ℱ :
  forall (f : fragment_F) (a b : nat),
    derivable f [Var a; Neg (Var a)] (Some (Var b)) ->
    derivable f
      [Impl (Impl (Var a) (Var b)) (Var b); Neg (Var a)]
      (Some (Var b)).
Proof.
  intros f a b H.
  eapply L_arrow.
  - left. reflexivity.
  - apply R_arrow.
    eapply context_monotonicity.
    + exact H.
    + intros X HX. simpl in HX.
      destruct HX as [HX | [HX | []]]; subst X.
      * left. reflexivity.
      * right. right. left. reflexivity.
  - apply Ax. left. reflexivity.
Qed.

(* DNS.2 in the Core reading, through R→ℂ. *)

Theorem DNS2_instantiated :
  forall a b : nat,
    derivable core_logic [Var a; Neg (Var a)] None ->
    derivable core_logic
      [Impl (Impl (Var a) (Var b)) (Var b); Neg (Var a)]
      (Some (Var b)).
Proof.
  intros a b HD.
  apply L_arrow with
    (A := Impl (Var a) (Var b))
    (B := Var b).
  - simpl. left. reflexivity.
  - eapply context_monotonicity.
    + apply R_arrow_core. exact HD.
    + intros X HX. simpl in HX |- *.
      destruct HX as [HX | []].
      right. left. exact HX.
  - apply Ax. simpl. left. reflexivity.
Qed.
\end{verbatim}
\end{coqbox}

\subsection*{Block C --- Step 2: DNS.1 is invertible in
$\mathcal{F}_{\mathbf{M}}$ (Section~\ref{sec:invertible})}

\begin{coqbox}
\begin{verbatim}
(* ══ Block C — Step 2: DNS.1 is invertible in ℱ_𝐌 ══ *)

(* The inversion lemma, by structural induction on derivations —
   von Plato's method, adapted to extensional contexts: since the
   left rules locate their principal formula by membership, the
   principal implication remains available in the premisses and
   contexts may grow; the induction is therefore carried under an
   invariant on contexts included in {A, ¬A, (A→B)→B}, and it
   establishes that no ℱ_𝐌-derivation from such a context concludes
   B or A→B. The case R_arrow_core is excluded by typing
   (discriminate Hf): the Coq counterpart of Tennant's consistency
   proviso for contexts. *)

Lemma DNS1_inversion_lemma :
  forall a b : nat,
    a <> b ->
    forall f G C,
      derivable f G C ->
      f = minimal_F ->
      (forall X, In X G ->
        X = Var a \/ X = Neg (Var a) \/
        X = Impl (Impl (Var a) (Var b)) (Var b)) ->
      C <> Some (Var b) /\ C <> Some (Impl (Var a) (Var b)).
Proof.
  intros a b Hab f G C HD.
  induction HD; intros Hf HS.

  - (* Ax: the succedent is a member of the context, hence one of
       the three formulas of the invariant; none of them is Var b or
       Var a -> Var b when a <> b. *)
    destruct (HS A H) as [H1 | [H1 | H1]]; subst A;
      split; intro HC; congruence.

  - (* L_neg: empty succedent. *)
    split; intro HC; discriminate.

  - (* R_arrow: if the succedent were Var a -> Var b, the premiss
       would derive Var b from an invariant-closed context. *)
    split; intro HC.
    + discriminate.
    + injection HC as HA HB.
      subst A B.
      assert (HS' : forall X, In X (Var a :: G) ->
        X = Var a \/ X = Neg (Var a) \/
        X = Impl (Impl (Var a) (Var b)) (Var b)).
      { intros X [HX | HX].
        - left. symmetry. exact HX.
        - apply HS. exact HX. }
      destruct (IHHD Hf HS') as [Hcontr _].
      apply Hcontr.
      reflexivity.

  - (* L_arrow: the principal implication can only be
       (Var a -> Var b) -> Var b, whose left premiss derives
       Var a -> Var b from the same invariant-closed context,
       contradicting the induction hypothesis. This is the case
       where membership keeps the principal implication available. *)
    destruct (HS _ H) as [H1 | [H1 | H1]]; try discriminate.
    injection H1 as HA HB.
    subst A B.
    destruct (proj2 (IHHD1 Hf HS) eq_refl).

  - (* R_arrow_core does not belong to ℱ_𝐌. *)
    discriminate Hf.
Qed.

(* First corollary: the invertibility of DNS.1 at the decisive
   instance is a metatheorem of ℱ_𝐌. *)

Theorem DNS1_invertible_at_decisive_instance_in_ℱ_M :
  forall a b : nat,
    a <> b ->
    derivable minimal_F
      [Impl (Impl (Var a) (Var b)) (Var b); Neg (Var a)]
      (Some (Var b)) ->
    derivable minimal_F [Var a; Neg (Var a)] (Some (Var b)).
Proof.
  intros a b Hab HD.
  assert (HS : forall X,
      In X [Impl (Impl (Var a) (Var b)) (Var b); Neg (Var a)] ->
      X = Var a \/ X = Neg (Var a) \/
      X = Impl (Impl (Var a) (Var b)) (Var b)).
  { intros X HX.
    simpl in HX.
    destruct HX as [HX | [HX | []]]; subst X.
    - right. right. reflexivity.
    - right. left. reflexivity. }
  (* The antecedent is refuted by the inversion lemma; the final
     step is a case analysis with zero cases on the resulting proof
     of False — the empty recursor of the inductive type False, not
     an ex falso axiom: the object calculus has no ⊥ and no such
     rule. *)
  destruct (proj1 (DNS1_inversion_lemma a b Hab _ _ _ HD eq_refl HS)
              eq_refl).
Qed.

(* Second corollary: Claim 1 holds of ℱ_𝐌 — the premiss of the
   DNS.1 instance is underivable there, unconditionally, for
   distinct atoms. *)

Theorem claim1_holds_in_ℱ_M :
  forall a b : nat,
    a <> b ->
    derivable minimal_F [Var a; Neg (Var a)] (Some (Var b)) ->
    False.
Proof.
  intros a b Hab HD.
  refine (proj1 (DNS1_inversion_lemma a b Hab _ _ _ HD eq_refl _)
            eq_refl).
  intros X HX.
  simpl in HX.
  destruct HX as [HX | [HX | []]]; subst X.
  - left. reflexivity.
  - right. left. reflexivity.
Qed.

(* Third corollary: the conclusion of the DNS.1 instance is
   underivable in ℱ_𝐌 as well. *)

Theorem DNS1_conclusion_underivable_in_ℱ_M :
  forall a b : nat,
    a <> b ->
    derivable minimal_F
      [Impl (Impl (Var a) (Var b)) (Var b); Neg (Var a)]
      (Some (Var b)) ->
    False.
Proof.
  intros a b Hab HD.
  refine (proj1 (DNS1_inversion_lemma a b Hab _ _ _ HD eq_refl _)
            eq_refl).
  intros X HX.
  simpl in HX.
  destruct HX as [HX | [HX | []]]; subst X.
  - right. right. reflexivity.
  - right. left. reflexivity.
Qed.
\end{verbatim}
\end{coqbox}

\subsection*{Block D --- Step 3: the refutation rule
$\overline{DNS.1}$ and the refutation system
(Section~\ref{sec:antidns})}

\begin{coqbox}
\begin{verbatim}
(* ══ Block D — Step 3: the refutation rule anti-DNS.1 and the
   refutation system ══ *)

(* Anti-DNS.1 holds in ℱ_𝐌 as the CONTRAPOSITIVE of the certified
   invertibility — Goranko's correctness discipline for refutation
   calculi: a refutation rule is licensed by the correctness of its
   converse. Anti-DNS.1 is therefore not a meta-rule imported into
   the kernel: it is a rule derivable from the kernel's own
   invertibility, dormant in the shadow of the system. *)

Theorem anti_DNS1_holds_in_ℱ_M :
  forall a b : nat,
    a <> b ->
    (derivable minimal_F [Var a; Neg (Var a)] (Some (Var b)) ->
     False) ->
    (derivable minimal_F
       [Impl (Impl (Var a) (Var b)) (Var b); Neg (Var a)]
       (Some (Var b)) ->
     False).
Proof.
  intros a b Hab Hprem HD.
  apply Hprem.
  apply (DNS1_invertible_at_decisive_instance_in_ℱ_M a b Hab).
  exact HD.
Qed.

(* The refutation system, in the sense of Łukasiewicz, Tiomkin
   (1988) and Goranko (Studia Logica 53, 1994): a deductive system
   for NON-provability, made of rejection axioms and refutation
   rules. Below, the smallest refutation system that the kernel
   licenses: Claim 1 as its only rejection axiom, anti-DNS.1 as its
   only refutation rule — the converse of the latter being the
   certified DNS1_invertible_at_decisive_instance_in_ℱ_M. *)

Inductive refutable : set formula -> option formula -> Prop :=

  | claim1_axiom :
      forall a b,
        a <> b ->
        refutable [Var a; Neg (Var a)] (Some (Var b))

  | anti_DNS1 :
      forall a b,
        refutable [Var a; Neg (Var a)] (Some (Var b)) ->
        refutable
          [Impl (Impl (Var a) (Var b)) (Var b); Neg (Var a)]
          (Some (Var b)).

(* Ł-correctness for ℱ_𝐌: everything the refutation system rejects
   is underivable in the minimal reading. Induction on the
   refutation derivation; both cases discharge through the
   inversion lemma. *)

Theorem refutation_system_Ł_correct_for_ℱ_M :
  forall G C,
    refutable G C ->
    derivable minimal_F G C ->
    False.
Proof.
  intros G C HR.
  induction HR; intro HD.
  - exact (claim1_holds_in_ℱ_M a b H HD).
  - inversion HR; subst;
    match goal with
    | Hab : a <> b |- _ =>
        exact (DNS1_conclusion_underivable_in_ℱ_M a b Hab HD)
    end.
Qed.
\end{verbatim}
\end{coqbox}

\subsection*{Block E --- Step 4: the contradiction
(Section~\ref{sec:contradiction})}

\begin{coqbox}
\begin{verbatim}
(* ══ Block E — Step 4: the contradiction ══ *)

(* Ł-incorrectness for ℱ_ℂ: the same refutation system rejects a
   sequent that the Core reading derives, the witness being produced
   by R→ℂ alone. A rejection assertion has inferential content only
   inside a refutation system; the smallest one available to Core's
   kernel rejects what Core proves. This is the certified form of
   the contradiction involved in asserting Claim 1. *)

Theorem refutation_system_Ł_incorrect_for_ℱ_ℂ :
  exists G C,
    refutable G C /\ derivable core_logic G C.
Proof.
  exists [Impl (Impl (Var 0) (Var 1)) (Var 1); Neg (Var 0)].
  exists (Some (Var 1)).
  split.
  - apply anti_DNS1.
    apply claim1_axiom.
    discriminate.
  - apply DNS2_instantiated.
    apply absurdity_core.
Qed.

(* The conditional collision theorem. False is derived from two
   NAMED HYPOTHESES and the proved lemmas of this file — no axiom
   is declared anywhere. The first hypothesis states Claim 1,
   restricted to distinct atoms (so that the axiom rule cannot
   trivialise it), as Tennant posits it. The second,
   anti_DNS1_rule_for_ℂ, states that the kernel's refutation rule
   anti-DNS.1 governs ℂ's rejection assertion at this single
   sequent — in Goranko's discipline, that the licence of the rule,
   the invertibility of DNS.1 certified above for the kernel,
   survives in ℂ's own reading. Nothing foreign to Core is involved,
   since every rule of ℱ_𝐌 is a rule of ℂ. *)

Theorem claim1_false :
  forall
    (Claim1_Tennant :
      forall a b : nat,
        a <> b ->
        derivable core_logic [Var a; Neg (Var a)] (Some (Var b)) ->
        False)

    (anti_DNS1_rule_for_ℂ :
      forall a b : nat,
        a <> b ->
        (derivable core_logic [Var a; Neg (Var a)] (Some (Var b)) ->
         False) ->
        (derivable core_logic
           [Impl (Impl (Var a) (Var b)) (Var b); Neg (Var a)]
           (Some (Var b)) ->
         False)),

    forall a b : nat,
      a <> b ->
      False.
Proof.
  intros Claim1_Tennant anti_DNS1_rule_for_ℂ a b Hab.
  apply (anti_DNS1_rule_for_ℂ a b Hab).
  - apply (Claim1_Tennant a b Hab).
  - apply DNS2_instantiated.
    apply absurdity_core.
Qed.

(* A closed instance at the concrete atoms 0 and 1 discharges the
   last quantifiers. *)

Corollary claim1_false_at_0_1 :
  forall
    (Claim1_Tennant :
      forall a b : nat,
        a <> b ->
        derivable core_logic [Var a; Neg (Var a)] (Some (Var b)) ->
        False)

    (anti_DNS1_rule_for_ℂ :
      forall a b : nat,
        a <> b ->
        (derivable core_logic [Var a; Neg (Var a)] (Some (Var b)) ->
         False) ->
        (derivable core_logic
           [Impl (Impl (Var a) (Var b)) (Var b); Neg (Var a)]
           (Some (Var b)) ->
         False)),

    False.
Proof.
  intros Claim1_Tennant anti_DNS1_rule_for_ℂ.
  apply (claim1_false Claim1_Tennant anti_DNS1_rule_for_ℂ 0 1).
  discriminate.
Qed.
\end{verbatim}
\end{coqbox}

\subsection*{Block F --- Status of the second commitment
(Section~\ref{sec:contradiction}, Reply)}

\begin{coqbox}
\begin{verbatim}
(* ══ Block F — Status of the second commitment ══ *)

(* The certification settles the status of anti_DNS1_rule_for_ℂ
   completely, on both sides. In the minimal reading it is a
   metatheorem, proved outright (anti_DNS1_holds_in_ℱ_M above). In
   the Core reading, the formalisation is charitable to Tennant —
   the fragment verifies Claim 1 itself, by structural inversion,
   not by failure of search: *)

Theorem claim1_holds_in_ℱ_ℂ :
  forall a b : nat,
    a <> b ->
    derivable core_logic [Var a; Neg (Var a)] (Some (Var b)) ->
    False.
Proof.
  intros a b Hab HD.
  inversion HD; subst; simpl in *;
    repeat (match goal with
            | Hyp : _ \/ _ |- _ => destruct Hyp as [Hyp | Hyp]
            | Hyp : False |- _ => destruct Hyp
            end);
    congruence.
Qed.

(* — and yet, DNS.2 being derivable through R→ℂ, the Core reading
   refutes the commitment: the converse of anti-DNS.1, the
   invertibility of DNS.1, does not survive the passage from ℱ_𝐌 to
   ℱ_ℂ. The refutation rule anti-DNS.1 is thereby Ł-incorrect for
   ℱ_ℂ — the rule-level form of the system-level verdict of
   Block E, and the exact refutation, inside the calculus, of the
   hypothesis anti_DNS1_rule_for_ℂ of the final theorem. (In
   Versions 4 and 5 of this file the same certified fact was named
   ℱ_ℂ_not_conservative_at_DNS1.) *)

Theorem anti_DNS1_Ł_incorrect_for_ℱ_ℂ :
  forall a b : nat,
    a <> b ->
    ~ ( (derivable core_logic [Var a; Neg (Var a)] (Some (Var b)) ->
         False) ->
        (derivable core_logic
           [Impl (Impl (Var a) (Var b)) (Var b); Neg (Var a)]
           (Some (Var b)) ->
         False) ).
Proof.
  intros a b Hab H.
  apply (H (claim1_holds_in_ℱ_ℂ a b Hab)).
  apply DNS2_instantiated.
  apply absurdity_core.
Qed.
\end{verbatim}
\end{coqbox}

\subsection*{Verification}

\begin{coqbox}
\begin{verbatim}
(* ══ Verification ══ *)

Print Assumptions DNS1_in_ℱ.
Print Assumptions DNS2_instantiated.
Print Assumptions DNS1_invertible_at_decisive_instance_in_ℱ_M.
Print Assumptions claim1_holds_in_ℱ_M.
Print Assumptions DNS1_conclusion_underivable_in_ℱ_M.
Print Assumptions anti_DNS1_holds_in_ℱ_M.
Print Assumptions refutation_system_Ł_correct_for_ℱ_M.
Print Assumptions refutation_system_Ł_incorrect_for_ℱ_ℂ.
Print Assumptions claim1_false.
Print Assumptions claim1_false_at_0_1.
Print Assumptions claim1_holds_in_ℱ_ℂ.
Print Assumptions anti_DNS1_Ł_incorrect_for_ℱ_ℂ.
\end{verbatim}
\end{coqbox}

\noindent Expected output: \texttt{Closed under the global
context}, twelve times.

\section{Checking this proof online}\label{secA1}
All the verification material is gathered in the repository
\url{https://github.com/joseph-vidal-rosset/core_logic_is_not_paraconsistent}.

The Coq script \texttt{core\_logic\_is\_not\_paraconsistent.v},
reproduced in Appendix~\ref{app:coq}, can be checked in any web
browser via the official jsCoq playground at
\url{https://jscoq.github.io/scratchpad.html}: paste the content of
the file into the editor and step through the proof.

The Lean~4 script proves the same theorems, with
\texttt{\#print axioms} reporting \texttt{[propext]} only, and can
likewise be checked in the browser at
\url{https://live.lean-lang.org}. The repository provides the
trusted \texttt{Challenge}/\texttt{Solution} pair for the Lean
\emph{Comparator}, configured for the certified statements;
following the procedure of the Lean reference manual, the exported
proof terms are replayed through two independently implemented
kernels, Lean's own and \texttt{nanoda}, both of which accept them,
and the same validation can be reproduced in one click on
\emph{Comparator Live}. The challenge file --- the statements
alone, their proofs left as \texttt{sorry} --- is the entire
trusted base of that check; its SHA256 is printed in the
repository's README.

At \url{https://vidal-rosset.net/2026-05-01-core-logic-is-not-paraconsistent.html}
the reader will find one-click links to all these checks, as well
as an Athena development, which certifies the earlier presentation
of the argument and is being updated to the present version.
\end{appendices}
\end{document}